\documentclass{amsart}
\usepackage{euscript}
\usepackage{amsmath}
\usepackage{amssymb}
\usepackage{amsfonts}
\usepackage{mathrsfs}
\usepackage{mathtools,graphicx, tikz, xcolor,verbatim, bm, enumerate, enumitem,babel}

\usepackage[latin1]{inputenc}
\usepackage{verbatim}
\usepackage{bm}
\usepackage[justification=centering]{subcaption}

\newtheorem{theorem}{Theorem}[section] 
\newtheorem{lemma}[theorem]{Lemma}     

\theoremstyle{definition}
\newtheorem{definition}[theorem]{Definition}
\newtheorem{example}[theorem]{Example}

\usepackage[outer=1in,marginparwidth=.75in]{geometry}
\usepackage{marginnote}
\usetikzlibrary{calc}
\usetikzlibrary{positioning}
\usetikzlibrary{patterns, patterns.meta}

\usetikzlibrary{shapes.geometric}
\usetikzlibrary{shapes.geometric}

\usepackage{color}
\usepackage[latin1]{inputenc}

\tikzstyle{square} = [shape=regular polygon, regular polygon sides=4, minimum size=1cm, draw, inner sep=0, anchor=south, fill=gray!30]
\tikzstyle{squared} = [shape=regular polygon, regular polygon sides=4, minimum size=1cm, draw, inner sep=0, anchor=south, fill=gray!60]
\tikzstyle{squaredd} = [shape=regular polygon, pattern={crosshatch}, regular polygon sides=4, minimum size=1cm, draw, inner sep=0, anchor=south]
\tikzstyle{squareddd} = [shape=regular polygon, regular polygon sides=4, minimum size=1cm, draw, inner sep=0, anchor=south, fill=gray!120]%

\newcommand{\C}{{\mathbb{C}}}
\newcommand{\Z}{{\mathbb{Z}}}
\newcommand{\Q}{{\mathbb{Q}}}
\newcommand{\N}{{\mathbb{N}}}

\begin{document}

\title{The Minimal Euclidean Function on the Gaussian Integers}
\author{Hester Graves}
\address{Institute for Defense Analyses\\ Center for Computing Sciences\\ Bowie, Maryland 20715, USA}
\email{hkgrave@super.org}

\maketitle

\begin{abstract}  In 1949, Motzkin proved that every Euclidean domain $R$ has a minimal Euclidean function, $\phi_R$.  
He showed that when $R = \Z$, the minimal function is $\phi_{\Z}(x) = \lfloor \log_2 |x| \rfloor$.  
For over seventy years, $\phi_{\Z}$ has been the only example of an explictly-computed minimal function in a number field.
We give the first explicitly-computed minimal function in a non-trivial number field, $\phi_{\Z[i]}$,  
which computes the length of the shortest possible $(1+i)$-ary expansion of any Gaussian integer. 
We also present an algorithm that uses $\phi_{\Z[i]}$ to compute minimal $(1+i)$-ary expansions of Gaussian integers. 
We solve these problems using only elementary methods.
\end{abstract}

\section{Introduction}

Given $a, b \in \mathbb{N}_0 = \{ 0, 1, 2, 3, \ldots\}$, $b \neq 0$, we can divide $a$ by $b$ to get a quotient $q$ and a remainder $r$, 
which allows us to write $a =qb +r$, with $r <b$.   
If $r > \frac{b}{2}$, we may rewrite $a$ as $(q+1)b + (r-b)$, so there exist some $q', r' \in \Z= \{ 0, \pm 1, \pm 2, \pm 3, \ldots \}$ such that
$a = q' b + r'$ and $|r'| \leq \frac{b}{2}$. 

An integral domain $R$ is a commutative ring with a multiplicative identity where $ab =0$ implies either $a =0$ or $b =0$.  
A domain $R$ is called a \textbf{Euclidean domain} if there is a function $f: R \setminus 0 \rightarrow \N_0$ such that for any $a, b \in R \setminus 0$, there exist
$q, r \in R$ such that $a = qb +r$ and either $r=0$ or $f(r) < f(b)$.  The function $f$ is a \textbf{Euclidean function} for $R$.  
The first paragraph shows that $f(x) = |x|$ is a Euclidean function for $\Z$, and that $f(x) = \lfloor \log_2 |x|  \rfloor$ is a strictly smaller one.

In 1949 Motzkin \cite{Motzkin} proved that all Euclidean domains have a \textbf{minimal Euclidean function}, 
the point-wise minimum of the domain's Euclidean functions, and that $\phi_{\Z}(x) = \lfloor \log_2 |x| \rfloor$.  
In the last seventy years, there have been no minimal functions computed for any non-trivial rings of integers of number fields.  
This paper's purpose is to ameliorate the situation by computing the minimal function for the Gaussian integers,
$\Z[i] = \{ a + b i: a, b \in \Z\} \subset \Q(i) \subset \C$. 

H.W. Lenstra Jr. lay the foundation for Theorem \ref{alg_theo} in 1974 when he proved that  
  \begin{equation}\label{Lenstra_result}
  \phi_{\Z[i]}^{-1}([0,n]) = B_n = \left \{ \sum_{j=0}^n v_j ( 1 + i)^j, v_j \in \{0, \pm 1, \pm i\} \right \}.
  \end{equation}
This is a direct analogue of Motzkin's formula, $\phi_{\Z} (x) = \log_2 |x|$, as 
 Lenstra's formula is one less than the minimal number of digits needed to write $a+bi$ in base $(1+i)$,
 and $\phi_{\Z}(x)$ is one less than the least number of digits necessary to write $x$ in base $2$.  
 One can read his original proof in \cite{Lenstra}, or an elementary, geometric proof in the accompanying paper \cite{Graves}.
 
 Unfortunately, the representation of $B_n$ in equation \ref{Lenstra_result} makes computing $\phi_{\Z[i]}$ tricky, even on the integers.
 Because $5 = 4 +1 = -(1+i)^4 +1$, it is clear that $5 \in B_4$, but 
 $5$ can also be written as 
 \[(2 + 2i) + 2 + (1-i) -i = -i \left ( (1+i)^3 + (1+i)^2 + (1+i) + 1 \right ) \in B_3,\]
 so finding the smallest $B_n$ an integer belongs to is not obvious.  To confirm that $\phi_{\Z[i]}(5)=3$, we would have to compute all the elements of $B_2$, and check that $5$ is not one of them.  
 With the following definition and subsequent theorem, however, computing $\phi_{\Z[i]}$ and finding the least $B_n$ is surprisingly easy.
 
 \begin{definition}  For $n \geq 0$, we define the sequence 
\begin{equation*}
w_n = 
\begin{cases}
2^{k+1} + 2^k & \text{if}\ n=2k\\
2^{k+2} & \text{if}\ n = 2k +1.
\end{cases}
\end{equation*}
\end{definition}
If $n \geq 0$, then $w_{n+2} = 2 w_n$ and if $k \geq 1$, then $w_{2k+1} - w_{2k} = w_{2k} - w_{2k-1} = 2^k$.  
 
 \begin{theorem}\label{alg_theo} Suppose that $a+bi \in \Z[i]$, that $m$ is the smallest integer such that
 $\frac{\max(|a|,|b|)}{2^j} \leq w_m -2$, and that $j$ is the largest integer such that $2^j$ divides both $a$ and $b$. Then  
 \begin{equation*}
\phi_{\Z[i]}(a+bi) = 
\begin{cases}
m+2j  & \text{if}\ \frac{|a| + |b|}{2^j} \leq w_{m+1}-3\\
m+2j+1 & \text{if}\ \frac{|a| + |b|}{2^j} > w_{m+1} -3.
\end{cases}
\end{equation*}
\end{theorem}
 We introduce directions to compute $\phi_{\Z[i]}(a+bi)$ that run particularly quickly on computers, as the implementation takes advantage of processors'
 underlying binary nature.
We also present an algorithm using $\phi_{\Z[i]}$ to find minimal $(1+i)$-ary expansions of any Gaussian integer.  
 Lastly, we provide step-by-step examples computing $\phi_{\Z[i]} (90 + 44i)$ in Section \ref{how_to} and a minimal $(1+i)$-ary expansion of $90 +44i$ in 
 Section \ref{finding_expansions}.

\section{Euclidean functions and the Gaussian integers, $\Z[i]$}


The sets $B_n$, defined in the introduction, have several nice properties.  They are closed under both complex conjugation and multiplication by elements of $B_0 = \{0, \pm 1, \pm i\}$.  
If $a+bi \in B_n$, then $(1+i)^j (a+bi) \in B_{n+j}$.  Similarly, if $2^j$ divides both $a$ and $b$ for some $a+bi \in B_n$, then 
$\frac{a}{2^j} + \frac{b}{2^j} i \in B_{n-2j}$. 
Our definition of $B_n$ does not, however, allow us to easily determine whether a given element $a+bi$ is a member of $B_n$.  

We use the sequence $w_n$, also defined in the introduction, to define the following `octagonal snowflakes,' which we will use to deliminate the shape of the $B_n$'s.  
We use $a|b$ to denote `$a$ divides $b$', we use $a^b \parallel c$ when $a^b |c$ and $a^{b+1} \nmid c$, and we use $(a,b)$ as 
shorthand for `the greatest common divisor of $a$ and $b$.'  If $x \in \Z[i]$, we denote the real and imaginary parts by $\text{Re}(x)$ and $\text{Im}(x)$, so that $x = \text{Re}(x) + \text{Im}(x) i$.

\begin{definition}\label{octagons}  For $n \geq 0$, we define 
\begin{equation*}
S_n: = \{ x+yi \in \Z[i]\setminus 0: 2 \nmid \gcd(x,y); |x|,|y| \leq w_n -2; |x| + |y| \leq w_{n+1} - 3 \}.
\end{equation*}
\end{definition}

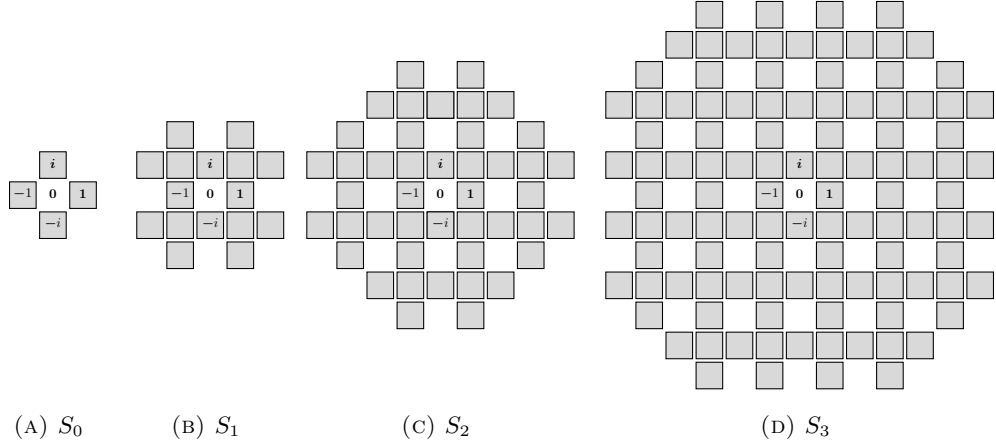
\begin{figure}[ht] \centering

\subcaptionbox{$S_0$}{
	\begin{tikzpicture} [scale=.5, transform shape]
		
		\draw[white] (-1.5, -5) -- (1.5, -5) -- (1.5,5.75) -- (-1.5, 5.75) -- (-1.5,-5);		
		\node[square]  at (.8,0) {};  
		\node[square]  at (-.8,0) {}; 
		\node[square]  at (0,.8) {}; 
		\node[square]  at (0,-.8) {}; 
		
		\node [circle,minimum size=1cm] at (0,.4) {$\bm 0 $};
		\node [circle,minimum size=1cm] at (.8,.4) {$\bm 1 $};
		\node [circle,minimum size=1cm] at (-.8,.4) {$\bm -1 $};
		\node [circle,minimum size=1cm] at (0,1.2) {$\bm i $};
		\node [circle,minimum size=1cm] at (0,-.4) {$\bm -i $};

	\end{tikzpicture}	}	
\subcaptionbox{$S_1$}{
	\begin{tikzpicture} [scale=.5, transform shape]
						
		\draw[white] (-1.5, -5) -- (1.5, -5) -- (1.5,5.75) -- (-1.5, 5.75) -- (-1.5,-5);		
		\foreach \y in {-2,...,2}
		\node[square]  at (.8,.8*\y) {}; 
		
		\foreach \y in {-2,...,2}
		\node[square]  at (-.8,.8*\y) {}; 
		
		\foreach \y in {-1,1}
		\node[square]  at (-1.6,.8*\y) {}; 
		
		\foreach \y in {-1,1}
		\node[square]  at (1.6,.8*\y) {}; 
		
		\foreach \y in {-1,1}
		\node[square]  at (0,.8*\y) {};

		\node [circle,minimum size=1cm] at (0,.4) {$\bm 0 $};
		\node [circle,minimum size=1cm] at (.8,.4) {$\bm 1 $};
		\node [circle,minimum size=1cm] at (-.8,.4) {$\bm -1 $};
		\node [circle,minimum size=1cm] at (0,1.2) {$\bm i $};
		\node [circle,minimum size=1cm] at (0,-.4) {$\bm -i $};

	\end{tikzpicture}			}	
\subcaptionbox{$S_2$}{
	\begin{tikzpicture} [scale=.5, transform shape]
						
		\draw[white] (-3.5, -5) -- (3.5, -5) -- (3.5,5.75) -- (-3.5, 5.75) -- (-3.5,-5);			
		\node[square]  at (.8,0) {};  
		\node[square]  at (-.8,0) {}; 
		\node[square]  at (0,.8) {}; 
		\node[square]  at (0,-.8) {};

		\node[square]  at (.8, .8) {}; 
		\node[square]  at (-.8, .8) {};
		\node[square]  at (-.8, -.8) {};
		\node[square]  at (.8, -.8) {};
		
		\node[square]  at (0, 2.4) {}; 
		
		\node[square]  at (.8, 1.6) {};
		\node[square]  at (.8, 2.4) {};
		\node[square]  at (.8, 3.2) {};
		
		\node[square]  at (1.6, .8) {}; 
		\node[square]  at (1.6, 2.4) {};
		
		\node[square]  at (2.4, .8) {};
		\node[square]  at (2.4, 1.6) {};
		
		\node[square]  at (3.2, .8) {};
		
		\node[square]  at (2.4, 0) {}; 
		
		\node[square]  at (0, -2.4) {}; 
		
		\node[square]  at (.8, -1.6) {};
		\node[square]  at (.8, -2.4) {};
		\node[square]  at (.8, -3.2) {};
		
		\node[square]  at (1.6, -.8) {}; 
		\node[square]  at (1.6, -2.4) {};
		
		\node[square]  at (2.4, -.8) {};
		\node[square]  at (2.4, -1.6) {};
		
		\node[square]  at (3.2, -.8) {};
		
		\node[square]  at (0, 2.4) {}; 
		
		\node[square]  at (-.8, 1.6) {};
		\node[square]  at (-.8, 2.4) {};
		\node[square]  at (-.8, 3.2) {};
		
		\node[square]  at (-1.6, .8) {}; 
		\node[square]  at (-1.6, 2.4) {};
		
		\node[square]  at (-2.4, .8) {};
		\node[square]  at (-2.4, 1.6) {};
		
		\node[square]  at (-3.2, .8) {};
		
		\node[square]  at (-2.4, 0) {}; 
		
		\node[square]  at (-.8, -1.6) {};
		\node[square]  at (-.8, -2.4) {};
		\node[square]  at (-.8, -3.2) {};
		
		\node[square]  at (-1.6, -.8) {}; 
		\node[square]  at (-1.6, -2.4) {};
		
		\node[square]  at (-2.4, -.8) {};
		\node[square]  at (-2.4, -1.6) {};
		
		\node[square]  at (-3.2, -.8) {};
		
		\node[square]  at (0, -.8) {};
		
		\node [circle,minimum size=1cm] at (0,.4) {$\bm 0 $};
		\node [circle,minimum size=1cm] at (.8,.4) {$\bm 1 $};
		\node [circle,minimum size=1cm] at (-.8,.4) {$\bm -1 $};
		\node [circle,minimum size=1cm] at (0,1.2) {$\bm i $};
		\node [circle,minimum size=1cm] at (0,-.4) {$\bm -i $};

	\end{tikzpicture}}	
	\subcaptionbox{$S_3$}{
	\begin{tikzpicture} [scale=.5, transform shape]
		
		\draw[white] (-5.5, -5) -- (5.5, -5) -- (5.5,5.75) -- (-5.5, 5.75) -- (-5.5,-5);
		\foreach \y in {-5,-3,-1, 1,3,5}
		\node[square]  at (0,.8*\y) {};
		
		\foreach \y in {-6,...,6}
		\node[square]  at (.8,.8*\y) {};
		
		\foreach \y in {-6,...,6}
		\node[square]  at (-.8,.8*\y) {};
		
		\foreach \y in {-5,-3,-1, 1,3,5}
		\node[square]  at (1.6,.8*\y) {};
		
		\foreach \y in {-5,-3,-1, 1,3,5}
		\node[square]  at (-1.6,.8*\y) {};
		
		\foreach \y in {-6,...,6}
		\node[square]  at (2.4,.8*\y) {};
		
		\foreach \y in {-6,...,6}
		\node[square]  at (-2.4,.8*\y) {};
		
		\foreach \y in {-5,-3,-1, 1,3,5}
		\node[square]  at (3.2,.8*\y) {};
		
		\foreach \y in {-5,-3,-1, 1,3,5}
		\node[square]  at (-3.2,.8*\y) {};
		
		\foreach \y in {-4,...,4}
		\node[square]  at (4,.8*\y) {};
		
		\foreach \y in {-4,...,4}
		\node[square]  at (-4,.8*\y) {};
		
		\foreach \y in {-3,-1, 1,3}
		\node[square]  at (-4.8,.8*\y) {};
		
		\foreach \y in {-3,-1, 1,3}
		\node[square]  at (4.8,.8*\y) {};

		\node [circle,minimum size=1cm] at (0,.4) {$\bm 0 $};
		\node [circle,minimum size=1cm] at (.8,.4) {$\bm 1 $};
		\node [circle,minimum size=1cm] at (-.8,.4) {$\bm -1 $};
		\node [circle,minimum size=1cm] at (0,1.2) {$\bm i $};
		\node [circle,minimum size=1cm] at (0,-.4) {$\bm -i $};

	\end{tikzpicture}}	
	\caption{Examples of $S_n$}
	\label{fig:first_octagons}
\end{figure}
Our sets are lacy and symmetrical 
 (see Figure \ref{fig:first_octagons}) 
 along the four lines $\text{Re}(x) =0$, $\text{Im}(x) = 0$, $\text{Re}(x) = \text{Im}(x)$,
and $\text{Re}(x) = - \text{Im}(x)$ because the $S_n$ are closed under complex conjugation and multiplication by units, 
like the sets 
$B_n$.  They are nested, so that $S_0 \subset S_1 \subset S_2$.  
The following theorem demonstrates the relationship between the sets $S_n$ and $B_n$.

\begin{theorem}\label{mainresult}
The set $B_n \setminus 0 $ is a disjoint union of multiples of the sets $S_{n-2j}$, where
\[B_n \setminus 0 =  \displaystyle \coprod_{j=0}^{\lfloor n/2 \rfloor } 2^j S_{n- 2j}.\]
\end{theorem}

\begin{proof}
We prove this by induction on $n$.  We first show that this holds when $n$ is even, and then we prove it for odd $n$.

Our base cases are $B_0 \setminus 0 = \{ \pm 1, \pm i \} $
and
$B_1 \setminus 0 = \{ \pm 1, \pm i, \pm 1 \pm i, \pm 2 \pm i, \pm 1 \pm 2i \}$.
Definition \ref{octagons} shows that $B_0 = S_0$ and $B_1 = S_1$.

\underline{\textbf{Case} $\mathbf{n =2k}$}:  We prove $B_{2k} = S_{2k}$ by showing containment in both directions. 
First suppose that $k \geq 1$ and that the theorem holds for all $j$, $0 \leq j < n=2k$.  
If $a+bi \in B_{2k} \setminus B_{2k-1}$, then there exists a unit $u \in \mathbb{Z}[i]^{\times}$ such that $(u(a+bi)-2^k) \in B_{2k-1}$,
so we may assume without loss of generality that $(a-2^k) + bi \in B_{2k-1}$.  
Our induction hypothesis implies that there exists some $j$, $0 \leq j < k$,
 such that $2^j \parallel (a-2^k, b)$,
\begin{align*}
 |a-2^k|, |b| \leq w_{2k-1} - 2^{j+1} &\text{ and } |a-2^k| + |b| \leq w_{2k} - 3\cdot 2^j.\\
\intertext{This means that }
|a|, |b| \leq w_{2k-1} + 2^k - 2^{j+1} &\text{ and }
|a|+|b| \leq w_{2k} + 2^k - 3 \cdot 2^j,\\
\intertext{and thus}
 |a|, |b| \leq w_{2k} - 2^{j+1}&\text{ and }
|a|+|b| \leq w_{2k+1} - 3 \cdot 2^j.
\end{align*}
As $2^j \parallel (a,b)$, $a+bi \in 2^j S_{2(k-j)}$ and $B_{2k}\setminus 0  \subset\displaystyle \bigcup_{j=0}^{\lfloor k \rfloor } 2^j S_{2(k-j)}$.

Now we prove containment in the other direction.  Suppose that $0 \leq j \leq k$ and that $a+bi \in \displaystyle 2^j S_{n- 2j}$. 
If $j > 0$, then $\frac{a}{2^j} + \frac{b}{2^j} i \in S_{2(k-j)}$, which is contained in $B_{2(k-j)}\setminus 0$ by our induction hypothesis, and thus $a+bi \in B_{2k}\setminus 0$.

If $j =0$, then $a + bi \in S_{2k}$.  Due to the symmetries of $S_{2k}$, we can assume without loss of generality that $a \geq b \geq 0$.  
Our induction hypothesis implies that $S_{2k-1} \subset B_{2k-1}\setminus 0$ and as $B_{2k-1} \subset B_{2k}$, we need only concern ourselves 
with $a + bi \in S_{2k} \setminus S_{2k-1}$.  
If $a+bi \in S_{2k} \setminus S_{2k-1}$, then either 
$a > w_{2k-1}-2 = 2^{k+1} - 2$ or $a+b > w_{2k}-3 = 2^{k+1} + 2^k - 3$.

In both of these situations, 
$a \geq 2^k$, so \[|a - 2^k| \leq w_{2k} - 2^k - 2 = w_{2k-1} -2.\] 
As $a + b \leq w_{2k+1} - 3$ and $0 \leq b \leq a$,  $2b \leq w_{2k+1} -4$, and thus $b \leq w_{2k-1} - 2$.  
Lastly, note 
$$|a-2^k| + |b| = a+b-2^k \leq w_{2k+1}-3-2^k =  w_{2k} -3.$$
The pair $a$ and $b$ are not both even, so $2 \nmid (a-2^k, b)$ and thus $(a-2^k) + bi \in S_{2k-1} \subset B_{2k-1}$.  Therefore,
$a+bi \in (B_{2k-1} + 2^k) \subset B_{2k}$, and $\displaystyle \bigcup_{j=0}^{\lfloor k \rfloor } 2^j S_{2(k-j)}$ is contained in $B_{2k}$.  
We have shown containment in both directions, so the two sets are, indeed, equal.

We assumed that $k\geq 1$ and that the theorem held for all $j$, $0 \leq j < 2k$.  After the first part of the proof, we can now say under the same assumptions, that the claim holds for all $j$, $0 \leq j \leq 2k$.

\underline{\textbf{Case} $\mathbf{n=2k+1}$}:  We will again prove this by containment in both directions.  
We just showed that $B_{2k} \subset \cup_{j=0}^k 2^j S_{2(k-j)} \subset \cup_{j=0}^k 2^j S_{2(k-j) +1}$.
Let $a+ bi \in B_{2k+1}\setminus B_{2k}$.  
We may again assume without loss of 
generality that $(a+bi) - 2^k(1+i) \in B_{2k}$ and that there exists some $j$, $0 \leq j \leq k$, such that $(a - 2^k) + (b- 2^k)i \in 2^j S_{2(k-j) }$.

Note that if $j=k$, then $|a-2^k| + |b-2^k| \leq 2^k$, so one of the summands must be zero.  
This implies that, even when $j =k$, $2^j \parallel (a,b)$.
We know that $|a-2^k|, |b-2^k| \leq w_{2k} - 2^{j+1}$ and $|a-2^k | + |b-2^k| \leq w_{2k+1} - 3 \cdot 2^j$.  
As $w_{2k} + 2^k = w_{2k+1}$ and $w_{2k+1} + 2^{k+1} =w_{2k+2}$, we see that 
\begin{align*}
2^j \parallel (a, b); |a|,|b| \leq w_{2k+1}-2^{j+1}; \text{ and } |a| + |b| \leq w_{2k+2} - 3\cdot 2^j,
\end{align*}
We infer that $B_{2k+1}\setminus B_{2k}$, and thus all of $B_{2k+1}$, is contained inside $\displaystyle \bigcup_{j=0}^{k } 2^j S_{2(k-j) +1}$.

To prove the other direction, let $0 \leq j \leq k$ and let $a+bi \in 2^j S_{2(k-j) +1}$.  
We can again assume without loss of generality that $a \geq b \geq 0$.  
When $j \geq 1$, we can apply our induction hypothesis to see that $\frac{a}{2^j} + \frac{b}{2^j} i \in S_{2(k-j) +1} \subset B_{2(k-j) +1}$, and thus $a + bi \in B_{2k + 1}$.

When $j =0$, we restrict ourselves once more to $a + bi \in S_{2k +1} \setminus S_{2k}$, as $S_{2k} \subset B_{2k} \subset B_{2k+1}$.  
If $a + bi \in S_{2k+1} \setminus S_{2k}$, then either $a>w_{2k} -2$ or $a+b > w_{2k+1} -3$.
We see that $a \geq 2^{k+1} -1$ in both situations, so $2^k -1 \leq a - 2^k \leq w_{2k}-2$ and $-2^k \leq b-2^k \leq w_{2k} -2$.  
If $b \geq 2^k$, then 
\begin{equation*}
 |a - 2^k| + |b-2^k| = a+b - 2^{k+1} \leq w_{2k+2} - 2^{k+1} - 3 = w_{2k+1} -3.
\end{equation*}
If $0< b < 2^k$, then 
\begin{equation*}
|a-2^k| + |b-2^k| \leq a-2^k + 2^k -b \leq a -1 \leq w_{2k+1} - 3.
\end{equation*}
Lastly, if $b=0$, then $a$ is odd, so 
\begin{equation*}
|a-2^k| + |b-2^k| = a \leq w_{2k+1} -3.
\end{equation*}
 
 In summary, if $j =0$ and $a+bi \in S_{2k+1} \setminus S_{2k}$, then $(a- 2^k) + (b - 2^k) i \in S_{2k} \subset B_{2k}$, so 
 \[a+bi = 2^k( 1 +i) + (a -2^k) + (b-2^k)i \in B_{2k+1}.\]  
 We conclude that $\displaystyle \bigcup_{j=0}^{k } 2^j S_{2(k-j) +1} \subset B_{2k+1} \setminus 0$, and the two sets are equal.  
\end{proof}

\begin{example} We see in Figure \ref{fig:B_2} that the set $B_2 \setminus 0$ is the union of $S_2$ (in light gray) and $2 S_0$ (in black). 
Similarly, the set $B_3 \setminus 0$ is the union of $S_3$ (in light gray) and $2 S_1$ (in black). 
\end{example}

\begin{figure}[ht]\centering 

\subcaptionbox{$B_2 \setminus 0 = S_2 \cup 2S_0$}{
	\begin{tikzpicture} [scale=.5, transform shape]
		\draw[white] (-5.5, -5) -- (5.5, -5) -- (5.5,5.75) -- (-5.5, 5.75) -- (-5.5,-5);	
		
		\foreach \y in {-3,-1, 1,3}
		\node[square]  at (0,.8*\y) {};
		
		\foreach \y in {-2,2}
		\node[squareddd]  at (0,.8*\y) {};
		
		\foreach \y in {-4,...,4}
		\node[square]  at (.8,.8*\y) {};
		
		\foreach \y in {-4,...,4}
		\node[square]  at (-.8,.8*\y) {};
		
		\foreach \y in {-3,-1,1,3}
		\node[square]  at (1.6,.8*\y) {};
		
		\foreach \y in {-3,-1,1,3}
		\node[square]  at (-1.6,.8*\y) {};
		
		\node[squareddd]  at (-1.6,0) {};
		\node[squareddd]  at (1.6,0) {};
		
		\foreach \y in {-2,...,2}
		\node[square]  at (2.4,.8*\y) {};
		
		\foreach \y in {-2,...,2}
		\node[square]  at (-2.4,.8*\y) {};
		
		\foreach \y in {-1,1}
		\node[square]  at (3.2,.8*\y) {};
		
		\foreach \y in {-1,1}
		\node[square]  at (-3.2,.8*\y) {};
		
		\node [circle,minimum size=1cm] at (0,.4) {$\bm 0 $};
		\node [circle,minimum size=1cm] at (.8,.4) {$\bm 1 $};
		\node [circle,minimum size=1cm] at (-.8,.4) {$\bm -1 $};
		\node [circle,minimum size=1cm] at (0,1.2) {$\bm i $};
		\node [circle,minimum size=1cm] at (0,-.4) {$\bm -i $};

\end{tikzpicture}}
\subcaptionbox{$B_3 \setminus 0 = S_3 \cup 2 S_1$}{
	\begin{tikzpicture} [scale=.5, transform shape]
		
		\draw[white] (-5.5, -5) -- (5.5, -5) -- (5.5,5.75) -- (-5.5, 5.75) -- (-5.5,-5);	
		\foreach \y in {-5,-3,-1, 1,3,5}
		\node[square]  at (0,.8*\y) {};
		
		\foreach \y in {-6,...,6}
		\node[square]  at (.8,.8*\y) {};
		
		\foreach \y in {-6,...,6}
		\node[square]  at (-.8,.8*\y) {};
		
		\foreach \y in {-5,-3,-1, 1,3,5}
		\node[square]  at (1.6,.8*\y) {};
		
		\foreach \y in {-5,-3,-1, 1,3,5}
		\node[square]  at (-1.6,.8*\y) {};
		
		\foreach \y in {-6,...,6}
		\node[square]  at (2.4,.8*\y) {};
		
		\foreach \y in {-6,...,6}
		\node[square]  at (-2.4,.8*\y) {};
		
		\foreach \y in {-5,-3,-1, 1,3,5}
		\node[square]  at (3.2,.8*\y) {};
		
		\foreach \y in {-5,-3,-1, 1,3,5}
		\node[square]  at (-3.2,.8*\y) {};
		
		\foreach \y in {-4,...,4}
		\node[square]  at (4,.8*\y) {};
		
		\foreach \y in {-4,...,4}
		\node[square]  at (-4,.8*\y) {};
		
		\foreach \y in {-3,-1, 1,3}
		\node[square]  at (-4.8,.8*\y) {};
		
		\foreach \y in {-3,-1, 1,3}
		\node[square]  at (4.8,.8*\y) {};
		
		\foreach \y in {-2,...,2}
		\node[squareddd]  at (1.6,1.6*\y) {};
		
		\foreach \y in {-2,...,2}
		\node[squareddd]  at (-1.6,1.6*\y) {};
		
		\foreach \y in {-2,2}
		\node[squareddd]  at (3.2,.8*\y) {};
		
		\foreach \y in {-2,2}
		\node[squareddd]  at (-3.2,.8*\y) {};
		
		\foreach \y in {-2,2}
		\node[squareddd]  at (0,.8*\y) {};

		\node [circle,minimum size=1cm] at (0,.4) {$\bm 0 $};
		\node [circle,minimum size=1cm] at (.8,.4) {$\bm 1 $};
		\node [circle,minimum size=1cm] at (-.8,.4) {$\bm -1 $};
		\node [circle,minimum size=1cm] at (0,1.2) {$\bm i $};
		\node [circle,minimum size=1cm] at (0,-.4) {$\bm -i $};

	\end{tikzpicture}}	
\caption{Examples of the sets $B_n \setminus 0$}
\label{fig:B_2}
\end{figure}
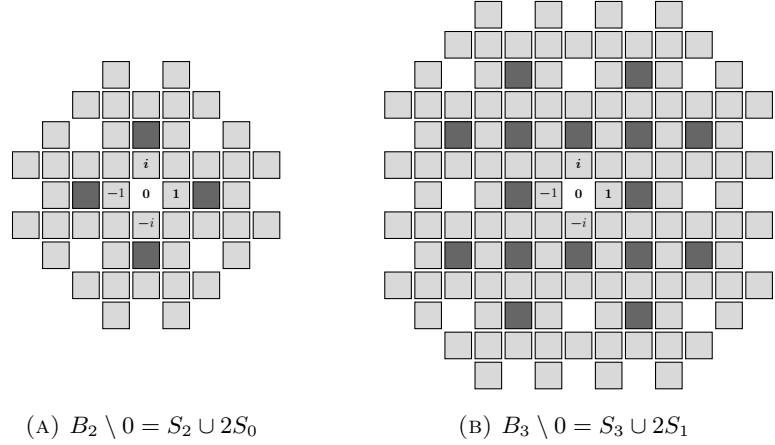

\subsection{An Alternate Formulation of our Main Result}

We can use another set of octagons that are even lacier than our sequence $S_n$.  

\begin{definition}\label{doilies_def}
For $n \geq 0$, we define 
\[D_n := \{ x +yi \in \Z[i]: 2 \nmid (x+y); |x|, |y| \leq w_n -2; |x| + |y| \leq w_{n+1} -3\}.\]
\end{definition}		

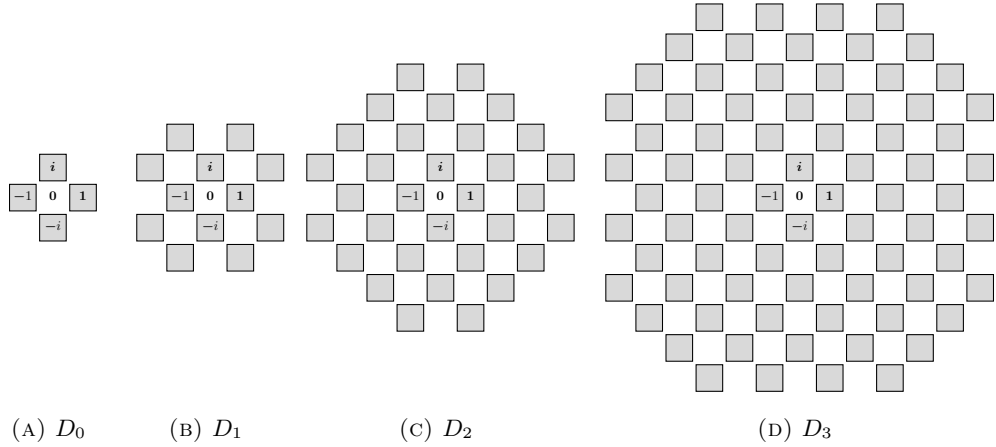
\begin{figure}[ht] \centering

\subcaptionbox{$D_0$}{
	\begin{tikzpicture} [scale=.5, transform shape]
		\draw[white] (-1.5, -5) -- (1.5, -5) -- (1.5,5.75) -- (-1.5, 5.75) -- (-1.5,-5);				
		\node[square]  at (.8,0) {};  
		\node[square]  at (-.8,0) {}; 
		\node[square]  at (0,.8) {}; 
		\node[square]  at (0,-.8) {}; 
		
		\node [circle,minimum size=1cm] at (0,.4) {$\bm 0 $};
		\node [circle,minimum size=1cm] at (.8,.4) {$\bm 1 $};
		\node [circle,minimum size=1cm] at (-.8,.4) {$\bm -1 $};
		\node [circle,minimum size=1cm] at (0,1.2) {$\bm i $};
		\node [circle,minimum size=1cm] at (0,-.4) {$\bm -i $}; 
		 
	\end{tikzpicture}	}	
\subcaptionbox{$D_1$}{
	\begin{tikzpicture} [scale=.5, transform shape]
		\draw[white] (-1.5, -5) -- (1.5, -5) -- (1.5,5.75) -- (-1.5, 5.75) -- (-1.5,-5);					
					
		\foreach \y in {-2,0,2}
		\node[square]  at (.8,.8*\y) {}; 
		
		\foreach \y in {-2,0,2}
		\node[square]  at (-.8,.8*\y) {}; 
		
		\foreach \y in {-1,1}
		\node[square]  at (-1.6,.8*\y) {}; 
		
		\foreach \y in {-1,1}
		\node[square]  at (1.6,.8*\y) {}; 
		
		\foreach \y in {-1,1}
		\node[square]  at (0,.8*\y) {};

		\node [circle,minimum size=1cm] at (0,.4) {$\bm 0 $};
		\node [circle,minimum size=1cm] at (.8,.4) {$\bm 1 $};
		\node [circle,minimum size=1cm] at (-.8,.4) {$\bm -1 $};
		\node [circle,minimum size=1cm] at (0,1.2) {$\bm i $};
		\node [circle,minimum size=1cm] at (0,-.4) {$\bm -i $};

	\end{tikzpicture}			}	
\subcaptionbox{$D_2$}{
	\begin{tikzpicture} [scale=.5, transform shape]
		\draw[white] (-3.5, -5) -- (3.5, -5) -- (3.5,5.75) -- (-3.5, 5.75) -- (-3.5,-5);					
					
		\foreach \y in {-3,-1, 1,3}
		\node[square]  at (0,.8*\y) {};
		
		\foreach \y in {-2,...,2}
		\node[square]  at (.8,1.6*\y) {};
		
		\foreach \y in {-2,...,2}
		\node[square]  at (-.8,1.6*\y) {};
		
		\foreach \y in {-3,-1, 1,3}
		\node[square]  at (1.6,.8*\y) {};
		
		\foreach \y in {-3,-1, 1,3}
		\node[square]  at (-1.6,.8*\y) {};
		
		\foreach \y in {-1,...,1}
		\node[square]  at (2.4,1.6*\y) {};
		
		\foreach \y in {-1,...,1}
		\node[square]  at (-2.4,1.6*\y) {};
		
		\node[square]  at (3.2,.8) {};
		
		\node[square]  at (-3.2,.8) {};
		
		\node[square]  at (3.2,-.8) {};
		
		\node[square]  at (-3.2,-.8) {};
		
		\node [circle,minimum size=1cm] at (0,.4) {$\bm 0 $};
		\node [circle,minimum size=1cm] at (.8,.4) {$\bm 1 $};
		\node [circle,minimum size=1cm] at (-.8,.4) {$\bm -1 $};
		\node [circle,minimum size=1cm] at (0,1.2) {$\bm i $};
		\node [circle,minimum size=1cm] at (0,-.4) {$\bm -i $};

	\end{tikzpicture}}	
	\subcaptionbox{$D_3$}{
	\begin{tikzpicture} [scale=.5, transform shape]
		\draw[white] (-5.5, -5) -- (5.5, -5) -- (5.5,5.75) -- (-5.5, 5.75) -- (-5.5,-5);	
		\foreach \y in {-5,-3,-1, 1,3,5}
		\node[square]  at (0,.8*\y) {};
		
		\foreach \y in {-3,...,3}
		\node[square]  at (.8,1.6*\y) {};
		
		\foreach \y in {-3,...,3}
		\node[square]  at (-.8,1.6*\y) {};
		
		\foreach \y in {-5,-3,-1, 1,3,5}
		\node[square]  at (1.6,.8*\y) {};
		
		\foreach \y in {-5,-3,-1, 1,3,5}
		\node[square]  at (-1.6,.8*\y) {};
		
		\foreach \y in {-3,...,3}
		\node[square]  at (2.4,1.6*\y) {};
		
		\foreach \y in {-3,...,3}
		\node[square]  at (-2.4,1.6*\y) {};
		
		\foreach \y in {-5,-3,-1, 1,3,5}
		\node[square]  at (3.2,.8*\y) {};
		
		\foreach \y in {-5,-3,-1, 1,3,5}
		\node[square]  at (-3.2,.8*\y) {};
		
		\foreach \y in {-2,...,2}
		\node[square]  at (4,1.6*\y) {};
		
		\foreach \y in {-2,...,2}
		\node[square]  at (-4,1.6*\y) {};
		
		\foreach \y in {-3,-1, 1,3}
		\node[square]  at (-4.8,.8*\y) {};
		
		\foreach \y in {-3,-1, 1,3}
		\node[square]  at (4.8,.8*\y) {};

		\node [circle,minimum size=1cm] at (0,.4) {$\bm 0 $};
		\node [circle,minimum size=1cm] at (.8,.4) {$\bm 1 $};
		\node [circle,minimum size=1cm] at (-.8,.4) {$\bm -1 $};
		\node [circle,minimum size=1cm] at (0,1.2) {$\bm i $};
		\node [circle,minimum size=1cm] at (0,-.4) {$\bm -i $};

	\end{tikzpicture}}	
	\caption{Examples of $D_n$}
	\label{fig:first_doilies}
\end{figure}

The set $D_n$ is the subset of elements of $S_n$ that are 
not divisible by $1+i$.  Like the $S_n$ and $B_n$, the sets $D_n$ are nested, and they are closed under
both complex conjugation and multiplication by units.

\begin{lemma}\label{snowflakes_and_doilies} For $n \geq 1$, 
$S_n = D_n \cup (1+i)D_{n-1}.$
\end{lemma}
\begin{proof} As $(1+i) | (x+yi)$ if and only if $2 | (x+y)$, it is clear that $\{x+yi \in S_n: (1+i) \nmid (x+yi) \} = D_n$.  To complete the proof, we must show that 
$\{x +yi \in S_n:(1+i) | (x+yi) \} = (1+i) D_{n-1}.$
We do so by showing containment in both directions.

Suppose $a+bi \in D_{n-1}$, noting that $(1+i) \nmid (a+bi)$.   
We can assume without loss of generality that $w_{n-1} - 2 \geq a > b \geq 0$.  
Note that 
$(1+i)(a+bi) = (a-b) +(a+b) i$, so that 
$0 < a-b \leq a+b \leq w_n -3<w_n -2$ and
\[ 0 < (a-b) + (a+b) = 2a \leq 2(w_{n-1} - 2) = w_{n+1} - 4 < w_{n+1} -3.\]
As $2 \nmid (a+b)$, we see $2 \nmid \gcd(a-b, a+b)$, and therefore infer that $(a+bi)(1+i) \in S_n$.

Now suppose $a+bi \in S_n \setminus 0$, with $(1+i) | (a+bi)$.  We can assume without loss of generality that 
$a \geq b \geq 0$ and $a >0$.  Then $\frac{a+bi}{1+i} = \frac{(a+bi)(1-i)}{2} = \frac{a+b}{2} + \frac{b-a}{2} i$. 
As $(1+i) |(a+bi)$, we know that $2 | (a+b)$ and thus $0 \leq a+b \leq w_{n+1} - 4$, implying that 
$0 \leq \left | \frac{b-a}{2} \right | \leq \frac{a + b}{2} \leq w_{n-1} - 2$.
We can also see that 
$\left | \frac{a+b}{2} \right | + \left | \frac{b-a}{2} \right | = \frac{a+b + a -b}{2} = a \leq w_n -2$.
By definition, 
$2 \nmid \gcd(a,b)$, so the condition $2 | (a+b) $ means that $a$ must be odd.
Hence $\frac{a+b}{2} + \frac{b-a}{2}$ is odd and 
$\left | \frac{a+b}{2} \right | + \left | \frac{b-a}{2} \right | \leq w_n -3$, 
so $\frac{a+b}{2} + \frac{b-a}{2} i \in D_{n-1}$, allowing us to conclude that $a+bi \in (1+i) D_{n-1}$.
\end{proof}

\begin{example}  Figure \ref{fig:snow_and_doilies} shows that $S_3$ is the union of $D_3$ (in light grey) and $(1+i)D_2$ (in black).
\end{example}
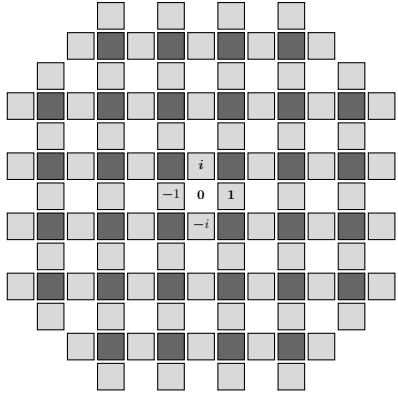
\begin{figure}[ht] \centering

	\begin{tikzpicture} [scale=.5, transform shape]
		
		\foreach \y in {-5,-3,-1, 1,3,5}
		\node[square]  at (0,.8*\y) {};
		
		\foreach \y in {-3,...,3}
		\node[square]  at (.8,1.6*\y) {};
		
		\foreach \y in {-3,...,3}
		\node[square]  at (-.8,1.6*\y) {};
		
		\foreach \y in {-5,-3,-1, 1,3,5}
		\node[square]  at (1.6,.8*\y) {};
		
		\foreach \y in {-5,-3,-1, 1,3,5}
		\node[square]  at (-1.6,.8*\y) {};
		
		\foreach \y in {-3,...,3}
		\node[square]  at (2.4,1.6*\y) {};
		
		\foreach \y in {-3,...,3}
		\node[square]  at (-2.4,1.6*\y) {};
		
		\foreach \y in {-5,-3,-1, 1,3,5}
		\node[square]  at (3.2,.8*\y) {};
		
		\foreach \y in {-5,-3,-1, 1,3,5}
		\node[square]  at (-3.2,.8*\y) {};
		
		\foreach \y in {-2,...,2}
		\node[square]  at (4,1.6*\y) {};
		
		\foreach \y in {-2,...,2}
		\node[square]  at (-4,1.6*\y) {};
		
		\foreach \y in {-3,-1, 1,3}
		\node[square]  at (-4.8,.8*\y) {};
		
		\foreach \y in {-3,-1, 1,3}
		\node[square]  at (4.8,.8*\y) {};
		
		\foreach \y in {-5,-3, -1, 1, 3, 5}
		\node[squareddd]  at (.8,.8*\y) {};
		
		\foreach \y in {-5,-3, -1, 1, 3, 5}
		\node[squareddd]  at (-.8,.8*\y) {};
		
		\foreach \y in {-5,-3, -1, 1, 3, 5}
		\node[squareddd]  at (2.4,.8*\y) {};
		
		\foreach \y in {-5,-3, -1, 1, 3, 5}
		\node[squareddd]  at (-2.4,.8*\y) {};
		
		\foreach \y in {-3, -1, 1, 3}
		\node[squareddd]  at (4,.8*\y) {};
		
		\foreach \y in {-3, -1, 1, 3}
		\node[squareddd]  at (-4,.8*\y) {};

		\node [circle,minimum size=1cm] at (0,.4) {$\bm 0 $};
		\node [circle,minimum size=1cm] at (.8,.4) {$\bm 1 $};
		\node [circle,minimum size=1cm] at (-.8,.4) {$\bm -1 $};
		\node [circle,minimum size=1cm] at (0,1.2) {$\bm i $};
		\node [circle,minimum size=1cm] at (0,-.4) {$\bm -i $};

	\end{tikzpicture}
	\caption{$S _3 = D_3 \cup (1+i) D_2$}
	\label{fig:snow_and_doilies}
\end{figure}

Lemma \ref{snowflakes_and_doilies} allows us to restate Theorem \ref{mainresult} in terms of the $D_n$, rather than the $S_n$.

\begin{theorem} \label{doily_mainresult}   The set $B_n\setminus 0$ is a disjoint union of multiples of sets $D_n$, where
\[B_n = \coprod_{j=0}^n (1+i)^j D_{n-j}.\]
\end{theorem}
\begin{proof}  Theorem \ref{mainresult} states that $B_n \setminus 0 = \bigcup_{j=0}^{\lfloor n/2 \rfloor} 2^j S_{n-2j}$.  
If $n =2j$, then $2^j S_0 = (1+i)^{2j} D_0$. 
As $2^j S_{n-2j} = (1+i)^{2j} D_{n-2j} \cup (1+i)^{2j+1} D_{n - (2j +1)}$ when $n - 2j \geq 0$,
the rest follows.
\end{proof}
 
\begin{example} In Figure \ref{fig:B_doily_union}, $B_3$ is the union of $D_3$ (in lightest grey), $(1+i)D_2$ (in medium grey), $(1+i)^2 D_1$ (in hatched grey), and $(1+i)^3 D_0$ (in black).
\end{example}
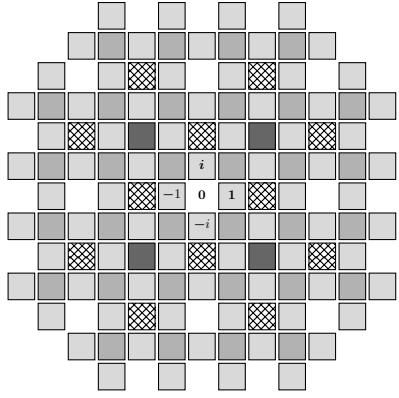
\begin{figure}[ht] \centering

	\begin{tikzpicture} [scale=.5, transform shape]
		
		\foreach \y in {-5,-3,-1, 1,3,5}
		\node[square]  at (0,.8*\y) {};
		
		\foreach \y in {-3,...,3}
		\node[square]  at (.8,1.6*\y) {};
		
		\foreach \y in {-3,...,3}
		\node[square]  at (-.8,1.6*\y) {};
		
		\foreach \y in {-5,-3,-1, 1,3,5}
		\node[square]  at (1.6,.8*\y) {};
		
		\foreach \y in {-5,-3,-1, 1,3,5}
		\node[square]  at (-1.6,.8*\y) {};
		
		\foreach \y in {-3,...,3}
		\node[square]  at (2.4,1.6*\y) {};
		
		\foreach \y in {-3,...,3}
		\node[square]  at (-2.4,1.6*\y) {};
		
		\foreach \y in {-5,-3,-1, 1,3,5}
		\node[square]  at (3.2,.8*\y) {};
		
		\foreach \y in {-5,-3,-1, 1,3,5}
		\node[square]  at (-3.2,.8*\y) {};
		
		\foreach \y in {-2,...,2}
		\node[square]  at (4,1.6*\y) {};
		
		\foreach \y in {-2,...,2}
		\node[square]  at (-4,1.6*\y) {};
		
		\foreach \y in {-3,-1, 1,3}
		\node[square]  at (-4.8,.8*\y) {};
		
		\foreach \y in {-3,-1, 1,3}
		\node[square]  at (4.8,.8*\y) {};
		
		\foreach \y in {-5,-3, -1, 1, 3, 5}
		\node[squared]  at (.8,.8*\y) {};
		
		\foreach \y in {-5,-3, -1, 1, 3, 5}
		\node[squared]  at (-.8,.8*\y) {};
		
		\foreach \y in {-5,-3, -1, 1, 3, 5}
		\node[squared]  at (2.4,.8*\y) {};
		
		\foreach \y in {-5,-3, -1, 1, 3, 5}
		\node[squared]  at (-2.4,.8*\y) {};
		
		\foreach \y in {-3, -1, 1, 3}
		\node[squared]  at (4,.8*\y) {};
		
		\foreach \y in {-3, -1, 1, 3}
		\node[squared]  at (-4,.8*\y) {};

		\foreach \y in {-2,0,2}
		\node[squaredd]  at (1.6,1.6*\y) {}; 
		
		\foreach \y in {-2,0,2}
		\node[squaredd]  at (-1.6,1.6*\y) {}; 
		
		\foreach \y in {-1,1}
		\node[squaredd]  at (-3.2,1.6*\y) {}; 
		
		\foreach \y in {-1,1}
		\node[squaredd]  at (3.2,1.6*\y) {}; 
		
		\foreach \y in {-1,1}
		\node[squaredd]  at (0,1.6*\y) {};

		\node[squareddd]  at (1.6,1.6) {}; 
		\node[squareddd]  at (1.6,-1.6) {};
		\node[squareddd]  at (-1.6,1.6) {};
		\node[squareddd]  at (-1.6,-1.6) {};

		\node [circle,minimum size=1cm] at (0,.4) {$\bm 0 $};
		\node [circle,minimum size=1cm] at (.8,.4) {$\bm 1 $};
		\node [circle,minimum size=1cm] at (-.8,.4) {$\bm -1 $};
		\node [circle,minimum size=1cm] at (0,1.2) {$\bm i $};
		\node [circle,minimum size=1cm] at (0,-.4) {$\bm -i $};

	\end{tikzpicture}
	\caption{$B _3 = D_3 \cup (1+i) D_2 \cup 2 D_1 \cup 2(1+i)D_0$}
	\label{fig:B_doily_union}
\end{figure}

\section{Computing the minimal Euclidean function on the Gaussian integers}\label{how_to}

Lenstra's Theorem tells us that 
\[
\phi_{\Z[i]}^{-1}(n) = \phi_{\Z[i]}^{-1}([0,n]) \setminus \phi_{\Z[i]}^{-1}([0,n-1])
=B_n \setminus B_{n-1}.\]
Therefore, to compute $\phi_{\Z[i]}(a + bi)$ for $a + bi \in \Z[i]\setminus 0$, 
we need to find the least value of $n$ such that $a + bi \in B_n$.  
We do this using Theorem \ref{mainresult}.  
Note $\phi_{\Z[i]}(\pm a \pm bi) = \phi_{\Z[i]}(\pm b \pm ai)$ as the sets $B_n$ are closed under complex conjugation and multiplication by units.

We can now prove Theorem \ref{alg_theo}.
\begin{proof}   
If $\frac{|a| + |b|}{2^j} \leq w_{m+1} - 3$, then $a + bi \in 2^j ( S_m \setminus S_{m-1})$ and $a + bi \in B_{m+2j} \setminus B_{m + 2j -1}$ by 
Theorem \ref{mainresult}.  
Suppose that $\frac{|a| + |b|}{2^j} > w_{m+1} - 3$ and $a + bi \notin 2^j S_m$
Our assumption $a\geq b \geq 0$ tells us that 
$\frac{|a| + |b|}{2^j} \leq 2 (w_m -2) \leq w_{m +2} - 4$, so $a+bi \in 2^j (S_{m+1}\setminus  S_m)$.  We deduce that $a + bi \in B_{m + 2j+1} \setminus B_{m + 2j}$ by Theorem \ref{mainresult}.
\end{proof}

Theorem \ref{alg_theo} gives us a clear algorithm to compute $\phi_{\Z[i]}(a+bi)$ for any $a +bi \in \Z[i]\setminus 0$, as $\phi_{\Z[i]}(a+bi) = \phi_{\Z[i]} (\max(|a|, |b|) + \min(|a|, |b|)i )$.
It runs rapidly on computers, as they all run on binary arithmetic.  
We can use the bitwise operations $\&$ and $>>$ to quickly find the greatest power of $2$ that divides any non-zero integer $x$. 
We outline the algorithm below.

\begin{enumerate}
\item Given $\alpha +\beta i$, let $a = \max (|\alpha|, |\beta|)$, $b = \min (|\alpha|, |\beta|)$.  \\
\item 
Let $j$ be the integer satisfying $2^j \parallel (a,b)$.\\
\item  
Let $p = \lfloor \log_2 | \frac{a}{2^j} + 2 | \rfloor.$ \\
When $p =1$:\\
\begin{itemize}
\item If $b =0$, then $m = 0$.\\
\item If $b =2^j$, then $M =1$.\\
\end{itemize}
When $p \geq 2$:\\
\begin{itemize}
\item If $(\frac{a}{2^j} + 2) = 2^p$, let $m = 2p - 3$.\\
\item If $2^p < (\frac{a}{2^j} + 2) \leq 2^p + 2^{p-1}$, then let $m = 2p - 2$.\\
\item If $ 2^p + 2^{p-1} < (\frac{a}{2^j} + 2) < 2^{p+1}$, then let $m = 2p - 1$.\\
\end{itemize}
\item If $\frac{a + b}{2^J} +3  \leq w_{m+1}$, then $\phi_{\Z[i]}(\alpha+\beta i) = m + 2j$.  Otherwise, $\phi_{\Z[i]}(\alpha+\beta i) = m + 2j +1$.
\end{enumerate}

\begin{example}\label{computation}  Let $a + bi = 90 + 44i$.  Note that $a + bi = 2(45 + 22i)$, so $J = 1$.  Then $\frac{A}{2^J} = 45$ and $32 < 45 + 2 < 48$, so $P=5$ and $M=8$.
Next, we see that $45 + 22 + 3 = 70 >64 = w_9$, so $\phi(90 + 44i) = 8 + 2 + 1 =11.$
\end{example}

\section{Finding minimal $(1 +i)$-ary expansions}\label{finding_expansions}

Section \ref{how_to} gives an algorithm to find the length of a minimal $(1+i)$-ary expansion of $a +bi$, $\phi_{\Z[i]}(a + bi)$.  
We now use $\phi_{\Z[i]}$ to then find one of those minimal expansions.
 The naive way to find a minimal expansion $\sum_{j=0}^n u_j (1 +i)^j, u_j \in \{0, \pm 1, \pm i \}$ of $a+bi$ would be to compute the four values
 $\phi_{\Z[i]} (a + bi - w(1 +i)^n)$ for $w \in \{\pm 1, \pm i\}$.  
 For at least one of the $w \in \{ \pm 1, \pm i \}$, $\phi_{\Z[i]} (a + bi - w(1 +i)^n) < n$.  We set $u_n = w$ for one of the obliging values of $w$, 
 and then repeat with the difference $a + bi - w(1 +i)^n$. 
  With a little forethought, however, we don't have to compute all four differences.  
  Lemmas \ref{even_expansion} and \ref{odd_expansion} below
  show that if we check the signs of 
 $a$ and $b$ and compare their magnitudes, a coefficient $u_n$ presents itself.  After Lemma \ref{odd_expansion}, we present an example applying this technique.

\begin{lemma}\label{even_expansion}
If $a + bi \in \Z[i] \setminus 0$ and $\phi_{\Z[i]}(a+bi) = 2k$, then there exists a $(1 +i)$-ary expansion of $a+bi$ of length $2k+1$ where 
\[u_{2k}(1 +i)^{2k} =  
\begin{cases}
2^k & \text{if } 0 \leq |b| \leq a \\
- 2^k & \text{if }  0 \leq |b| \leq -a \\
-2^k i & \text{if }  0 \leq |a| \leq -b \\
2^k i & \text{if } 0 \leq |a| \leq b
\end{cases} 
   .\]
\end{lemma}

\begin{proof} Let us suppose that $a \geq |b| \geq 0$ and that $2^j \parallel (a,b)$, with $0 \leq j \leq k$, so that $a+bi \in 2^j S_{2(k-j)}$.

If $k=j$, then $\frac{a}{2^j} + \frac{b}{2^j} i \in S_0$, and $a +bi = 2^k$, so $(a - 2^k) + bi =0$ and $u_{2k}(1+i)^{2k} = 2^k$.

Now suppose that $0 \leq j <k$.  As $\frac{2|b|}{2^j} \leq \frac{|a| + |b|}{2^j} \leq w_{2(k-j) +1} - 3$, we see that $\frac{|b|}{2^j} \leq w_{2(k-j) -1} - 2$.  
The element $a + bi \notin 2^j S_{2(k-j) -1}$, 
so either 
\[\frac{|a|}{2^j} > w_{2(k-j) -1} - 2 \geq w_{2(k-j-1)} = 2^{k-j} + 2^{k-j-1} > 2^{k-j},\]
or $\frac{|a| + |b|}{2^j} > w_{2(k-j)} - 3$.  In the second scenario,
\begin{align*}
\frac{|a|}{2^j} &> w_{2(k-j)} - 3 - \frac{|b|}{2^j} \geq w_{2(k-j)} - 3 + 2 - w_{2(k-j-1)},\\
\intertext{so}
\frac{|a|}{2^j} &\geq w_{2(k-j)} - w_{2(k-j-1)} = 2^{k-j}.
\end{align*}
In both situations, $a \geq 2^k$, so  $0 <  \frac{a - 2^k}{2^j} \leq w_{2(k-j) -1} -2$ and $0 < \frac{a + |b| - 2^k}{2^j} \leq w_{2(k-j)} - 3$.  
The assumption $j < k$ means that $2^j \parallel (a - 2^k, b)$,
so $(a-2^k) + bi \in 2^j S_{2(k-j)-1} \subset B_{2k -1}$,  and we can write $a +bi$ as $2^k$ plus some element of $B_{2k-1}$.  

If $0 \leq |b| \leq -a$, then $-(a+bi)$ is in the analyzed situation, so $-(a+bi) +2^k \in B_{2k -1}$, 
and thus $(a + bi) \in -2^k + B_{2k-1}$.  The other two claims follow from analogous reasoning.
\end{proof}

\begin{lemma}\label{odd_expansion}
If $a + bi \in \Z[i] \setminus 0$ and $\phi(a+bi) = 2k +1$, then there exists a $(1 +i)$-ary expansion of $a+bi$ of length $2k+2$, where 
\[u_{2k+1}(1 +i)^{2k+1} = 
\begin{cases}
2^k (1 +i) & \text{if } a, b \geq 0 \\
- 2^k(1 + i) & \text{if }  a,b \leq 0\\
2^k (1 -i) & \text{if }  a \geq 0, b \leq 0 \\
-2^k (1 -i) & \text{if } a \leq 0, b \geq 0
\end{cases} 
   .\]
\end{lemma}

\begin{proof}  Suppose that $a,b \geq 0$ and that $2^j \parallel (a,b)$, with $0 \leq j \leq k$, so that $a+bi$ is in the first quadrant of $2^j S_{2(k-j) +1}$.  
If $j =k$, then $\frac{a}{2^k} + \frac{b}{2^k} i \in S_1 \setminus S_0$, 
and $a +bi \in 2^k \{1 +i, 2 +i, 1 +2i\} $, revealing that $(a -2^k) + (b-2^k) i \in  \{0, 2^k, 2^ki \} \subset B_{2k}$.

If $0 \leq j <k$, then 
$0 \leq \frac{a}{2^j},  \frac{b}{2^j}  \leq w_{2(k-j) + 1} -2,$
and 
\[\left | \frac{a - 2^k}{2^j} \right |, \left | \frac{b - 2^k}{2^j} \right |  \leq \max \{ 2^{k-j}, w_{2(k-j) + 1} - 2 - 2^{k-j} \}
 \leq w_{2(k-j)} - 2.\]
 If $a, b \geq 2^k$, then 
\[
\frac{|a - 2^k| + |b - 2^k|}{2^j} = \frac{a + b - 2^{k+1}}{2^j} \leq w_{2(k-j) + 2} - 3 - 2^{(k-j) +1} 
= w_{2(k-j) + 1} - 3.\]
As $a + bi \notin 2^j S_{2(k-j)}$, we cannot have  both $a, b \leq 2^k$.  Suppose exactly one of the coordinates is $> 2^k$, so that 
$|a - 2^k| + |b-2^k| = \max (a,b) - \min(a,b)$.  If $\min (a,b) =0$, then $2 \nmid \frac{\max(a,b)}{2^j}$, so 
$\frac{\max(a,b) - \min(a,b)}{2^j} \leq w_{2(k-j) +1} -3.$
We conclude that $(a - 2^k) + (b-2^k) \in 2^j S_{2(k-j)} \subset B_{2k}$.

As in Lemma \ref{even_expansion}, the other cases come from looking at associates of $a +bi$.

\end{proof}  

\begin{example}  We return to Example \ref{computation}.  We saw that $\phi_{\Z[i]}(90 + 44i) =11$, so Lemma \ref{odd_expansion} tells us that $u_{11}(1 +i)^{11} = 32(1+i)$.
Then:\\
\begin{itemize}
\item $\phi_{\Z[i]} (58 + 12i) = 9$, so $u_9(1 +i)^9 = 16(1 +i)$\\  
\item $\phi_{\Z[i]} (42 - 4i) = 8$, so $u_8(1 +i)^8 = 16$  \\
\item $\phi_{\Z[i]} (26 - 4i) = 7$, so $u_7(1 +i)^7 = 8(1 -i)$\\
\item $\phi_{\Z[i]}(18 + 4i) = 6$, so $u_6(1+i)^6 = 8$\\
\item $\phi_{\Z[i]}(10 + 4i) = 5$, so $u_5(1 +i)^5 =4(1 +i)$\\
\item $\phi_{\Z[i]}(6) =4$, so $u_4 (1 +i)^4 = 4$\\
\item $u_2(1+i)^2 = 2$.  
\end{itemize}
We can then write 
\[90 + 44i = -i(1+i)^{11}  + (1 +i)^9 + (1+i)^8 + (1 +i)^7 +i(1 +i)^6 -i(1+i)^5 - (1 +i)^4 -i(1 +i)^2,\]
where $u_{11} = u_5=  u_2= -i$, $u _9 =u_8= u_7= 1$, $u_6 =i$, $u_4 = -1$, and $u_{10} =u_3 = u_1 =0.$
\end{example}

\section{Further Work}

In \cite{Graves}, I use the geometry of the $B_n$ to give a new proof of Lenstra's theorem in $\Z[i]$.  
Lindsey-Kay Lauderdale, Ryan Keck, and I have another paper exploring
the minimal Euclidean function on the Eisenstein integers, but we do not have an explicit algorithm, as in the Gaussians.  
Tom Edgar and I have forthcoming research on alternate expansions of Gaussian integers.

\section*{Acknowledgements}
I would like to thank my very patient spouse, Loren LaLonde, 
who has listened to me talk about this problem for the last fifteen years, and who has ensured that my LaTeX always compiled.

I first encountered this problem while working on my Ph.D. thesis fifteen years ago, while under Nick Ramsey's kind supervision.  
In the intervening years, I sustained a traumatic brain injury.  
I would like to thank Dr. Mary Lee Esty, whose neurotherapy returned my 
cognitive capacity to engage in research mathematics.  
The final writing occurred during the pandemic.  I would like to thank the Thurgood Marshall Childhood Development Center, and especially 
Ms. Joemese Malloy, for taking care of my child with love, and enabling me to concentrate on mathematics.

I greatly appreciate the help from Michael Bridgland, who not only read Martin Fuchs' thesis for me (I don't speak German), 
but who also asked a question that led to other work. 
I extended my heartfelt thanks to Katie Ahrens, Chris Bishop, and Jon Grantham for their useful suggestions on 
previous drafts. 

Lastly, I would like to thank both H.W. Lenstra, Jr.  and Franz Lemmermeyer for their help with this paper's background research and literature review, as well as their kind personal encouragement.    
I highly recommend Lemmermeyer's survey, ``The Euclidean Algorithm in Algebraic Number Fields,'' to anyone interested in the subject \cite{Lemmermeyer}.

\end{document}